\documentclass[11pt, leqno]{article}
\usepackage[utf8]{inputenc}
\usepackage{t1enc}
\usepackage{subfiles}
\usepackage{enumitem}
	\setlist[enumerate]{label={\normalfont(\alph*)}} % or \upshape
\usepackage{amsfonts}
\usepackage{amsthm}
\usepackage{amsmath}
\usepackage{amssymb}
\usepackage{amscd}
\usepackage{bbm}
\usepackage{multirow}
\usepackage{mathtools}
\usepackage{esint}
\usepackage[mathscr]{eucal}

\usepackage{titlefoot}
\usepackage[margin=3cm]{geometry}
\usepackage{indentfirst}
\usepackage{graphicx}
\usepackage{graphics}
\usepackage{lscape}
\usepackage{tikz-cd}
\usepackage{color}
\usepackage{pict2e}
\usepackage{epic}
\usepackage{epstopdf}
\usepackage{titlesec}
	\titleformat{\section}[block]{\Large\bfseries\filcenter}{\thesection}{1em}{}
\usepackage{commath}
\usepackage{float}
\usepackage{caption}
\usepackage{subcaption}
\usepackage{etoolbox}
\usepackage[affil-it]{authblk}
\usepackage{combelow}

\let\oldbibliography\thebibliography
\renewcommand{\thebibliography}[1]{%
  \oldbibliography{#1}%
  \setlength{\itemsep}{-.5pt}%
}

\usepackage[hidelinks]{hyperref}
\hypersetup{bookmarksopen=true} 
\usepackage{hypcap}

\graphicspath{{./Pictures/}}
\allowdisplaybreaks

\usepackage{algorithm}
\usepackage[noend]{algpseudocode}

\usepackage{listings}
\usepackage{fancybox}

\renewcommand*\thesection{\arabic{section}}
\swapnumbers
\numberwithin{equation}{section}
\numberwithin{table}{section}
\numberwithin{figure}{section}

\theoremstyle{plain}
\newtheorem{teo}[equation]{Theorem}
\newtheorem{lema}[equation]{Lemma}
\newtheorem{prop}[equation]{Proposition}

\theoremstyle{definition}
\newtheorem{ndef}[equation]{Definition}

\newtheorem{question}[equation]{Question}

\newtheorem{remark}[equation]{Remark}

\newcommand{\thistheoremname}{}
\newtheorem{genericthm}[equation]{\thistheoremname}
\newenvironment{para}[1]
  {\renewcommand{\thistheoremname}{#1}%
   \begin{genericthm}}
  {\end{genericthm}}

\newcommand{\thistheoremnames}{}
\newtheorem*{genericthms}{\thistheoremnames}
\newenvironment{para*}[1]
  {\renewcommand{\thistheoremnames}{#1}%
   \begin{genericthms}}
  {\end{genericthms}}

\expandafter\let\expandafter\oldproof\csname\string\proof\endcsname
\let\oldendproof\endproof
\renewenvironment{proof}[1][\proofname]{%
  \oldproof[\upshape \bfseries #1:]%
}{\oldendproof}
\makeatletter
\def\@makechapterhead#1{%
  \vspace*{50\p@}%
  {\parindent \z@ \raggedright \normalfont
    \interlinepenalty\@M
    \Huge\bfseries  \thechapter.\quad #1\par\nobreak
    \vskip 40\p@
  }}
\makeatother

\def \R {\mathbb{R}}
\def \Z {\mathbb{Z}}

\def \N{\mathbb{N}}
\def \D{\textup{D}}
\def \e{\varepsilon}
\def \d{\,\textup{d}}

\def \exc{\backslash}

\def \mc{\mathcal}

\def \hs{\hspace{0.5cm}}

\def \tp{\textup}
\def \mb{\mathbb}

\makeatletter
\DeclareFontFamily{OMX}{MnSymbolE}{}
\DeclareSymbolFont{MnLargeSymbols}{OMX}{MnSymbolE}{m}{n}
\SetSymbolFont{MnLargeSymbols}{bold}{OMX}{MnSymbolE}{b}{n}
\DeclareFontShape{OMX}{MnSymbolE}{m}{n}{
    <-6>  MnSymbolE5
   <6-7>  MnSymbolE6
   <7-8>  MnSymbolE7
   <8-9>  MnSymbolE8
   <9-10> MnSymbolE9
  <10-12> MnSymbolE10
  <12->   MnSymbolE12
}{}
\DeclareFontShape{OMX}{MnSymbolE}{b}{n}{
    <-6>  MnSymbolE-Bold5
   <6-7>  MnSymbolE-Bold6
   <7-8>  MnSymbolE-Bold7
   <8-9>  MnSymbolE-Bold8
   <9-10> MnSymbolE-Bold9
  <10-12> MnSymbolE-Bold10
  <12->   MnSymbolE-Bold12
}{}

\let\llangle\@undefined
\let\rrangle\@undefined
\DeclareMathDelimiter{\llangle}{\mathopen}%
                     {MnLargeSymbols}{'164}{MnLargeSymbols}{'164}
\DeclareMathDelimiter{\rrangle}{\mathclose}%
                     {MnLargeSymbols}{'171}{MnLargeSymbols}{'171}
\makeatother

\begin{document}

 \title{\LARGE \textbf{Numerical evidence towards a positive answer to
   Morrey's problem}}

\author[1]{{\Large Andr\'e Guerra}}
\author[2]{{\Large  Rita Teixeira da Costa\vspace{0.4cm}}}

\affil[1]{\small University of Oxford, Andrew Wiles Building, Woodstock Rd, Oxford OX2 6GG, United Kingdom \protect \\
{\tt{andre.guerra@maths.ox.ac.uk}}\vspace{1em}\ }

\affil[2]{\small 
University of Cambridge, Department of Pure Mathematics and Mathematical Statistics, Wilberforce Road, Cambridge CB3 0WA, United Kingdom
\protect \\
{\tt{rita.t.costa@dpmms.cam.ac.uk}}\
}

\date{}

\maketitle

\begin{abstract}
We report on numerical experiments suggesting that rank-one convexity
imples quasiconvexity in the planar case. We give a simple
heuristic explanation of our findings.
\end{abstract}

\unmarkedfntext{
\hspace{-0.78cm} \textbf{Keywords:} Quasiconvexity, Rank-one convexity, Gradient Young measures, Laminates.
}

% \keywords{\small \textbf{Keywords:} \dots \textbf{MSC classification:}
% \dots }
% \keywords{\small \textbf{Acknowledgements:} }

\vspace{0.2cm}

\section{Introduction}

An important problem in the vectorial Calculus of Variations is to
characterize the integrands $f\colon \R^{m\times n}\to \R$ for which
the functional
$$\mc F[u]\equiv \int_\Omega f(\D u(x))\d x, \qquad \tp{where } u\colon
\Omega\subseteq \R^n \to \R^m \tp{ and } m,n\geq 2,$$
is lower semicontinuous with respect to the weak
topology in an appropriate Sobolev space; this is the natural
condition for existence of minimizers through the Direct Method.

In his seminal work \cite{Morrey1952}, \textsc{Morrey}
recognized that the weak lower semicontinuity of $\mc F$ is essentially
equivalent to a weak notion of convexity, called quasiconvexity, on
$f$.
Despite many efforts in the last five decades, 
an explicit description of quasiconvex functions remains elusive:
for instance, there are fourth-order polynomials whose
quasiconvexity has been neither proved nor disproved.
Such a description would be relevant not only in the Calculus of
Variations but also in other areas of analysis \cite{Iwaniec2002,
  Kirchheim2016,  Milton2002, Tartar1979}.

Quasiconvexity has been mostly studied
in relation with polyconvexity \cite{Ball1977} and rank-one convexity;
these are respectively stronger and weaker notions that are much easier
to tackle \cite{Dacorogna2007}.
We will focus on the relation between quasiconvexity and
rank-one convexity. It is useful to consider certain classes of
measures that can be seen as being dual to
these notions: gradient Young
measures and laminates are, respectively, the probability measures that
satisfy Jensen's inequality with respect to quasiconvex and
rank-one convex functions; see
\cite{Muller1999a} for more details.

The following remains one of the main open problems in the Calculus of
Variations:

\begin{question}\label{qu:Morrey}
Are rank-one convex functions quasiconvex? Equivalently, let $\nu$ be a compactly supported gradient Young measure in  $\R^{m\times n}$; is $\nu$ a laminate?
\end{question}

Quesion \ref{qu:Morrey} is usually referred to as Morrey's problem.
It seems that \textsc{Morrey} himself was not sure about what the answer to
Question \ref{qu:Morrey} should be \cite{Morrey1952,Morrey1966}. A
fundamental example of \textsc{\v Sverák} \cite{Sverak1992a} shows that the answer is
negative if $m\geq 3$, $n\geq 2$ and more recently
\textsc{Grabovsky} \cite{Grabovsky2018} obtained
a different example when $m=8$ and $n=2$. \textsc{\v Sverák}'s example
is a polynomial of degree four; \textsc{Grabovsky}'s example, although
analytically more complicated, has the advantage of being
2-homogeneous and invariant under the right-action of $\tp{SO}(2)$.
Question \ref{qu:Morrey} remains open in the case of
low-dimensional targets, i.e.\ when $m=2$ and $n\geq 2$.
There is some partial evidence suggesting that the answer might be
positive in this case, see e.g.\ the
landmark results of \cite{Faraco2008, Harris2018, Kirchheim2008, Muller1999b}. However, and despite remarkable progress, it is by no means clear that
the answer should be positive even in low dimensions.

Since the analytic study of
quasiconvexity remains incredibly challenging,
it is natural to look for numerical evidence instead: that is the
goal of this note. In line with the
ideas from \cite{Sverak1992a}, we fix
a Lipschitz deformation $u\colon \mb T^n\to \R^m$ whose gradient has finite image
and we look for rank-one convex functions falsifying Jensen's
inequality with respect to $\D u$. Given such a deformation,
it suffices to consider the rank-one convex envelopes of functions of the form
$$
f(A)=
\begin{cases}
  g(A) & A \in [\D u]\\
  2 & \tp{otherwise }
\end{cases},
$$
where $[\D u]$ denotes the essential range of $\D u\in
L^\infty(\Omega,\R^{m\times n})$ and $g\colon [\D
u]\to [-1,1]$ is any function. In other words, for a deformation such
that $[\D u]$ is finite the task of looking for counterexamples is a
finite-dimensional problem. As a small technical remark we note that
it is important that $f$ only takes finite values; it is easy to build
examples of rank-one convex non-quasiconvex functions if the value
$+\infty$ is allowed \cite{Ball1990a}.

Arbitrary deformations can be approximated by linear
combinations of plane waves and we will consider finite sums of such
waves, see  Section
\ref{sec:deformations}.
Our choice is inspired by \textsc{James}'s interpretation of
\textsc{\v Sverák}'s example, where three waves suffice to build a counterexample.

Numerical searches
for counterexamples to Question \ref{qu:Morrey} have been undertaken before
\cite{Dacorogna1990a, Dacorogna1998, Gremaud1995}. The
strategy of these papers is the opposite from the one we pursue here: they fix a rank-one convex
function $f$ and look for deformations for which $f$ does
not satisfy Jensen's inequality. The main problem of this approach
is that explicit rank-one convex non-polyconvex
functions are rare and the available examples are relatively simple and have
many symmetries. Our procedure is much more general, since our
deformations are sampled randomly and the rank-one convex functions
have as little structure as possible.

By homogenization, the gradient of a Lipschitz
deformation $u\colon \mb T^n \to \R^m$ generates
a gradient Young measure, which has finite support if $[\D u]$ is
finite; thus our goal is to determine whether this measure is a
laminate. Hence, we are naturally led to consider:

\begin{question}[\cite{Kirchheim2003}]\label{qu:KMS}
Is there an effective algorithm to decide whether a given probability measure supported on a finite subset of $\R^{m\times n}$ is a laminate?
\end{question}

Deciding whether a given measure is a laminate is
difficult, as in principle one has to test Jensen's inequality with
all rank-one convex functions \cite{Pedregal1993}.
One possible way of circumventing this issue
is to consider just the extremal rank-one convex functions
\cite{Guerra2018}; however, the general structure of these functions
remains unclear. A different approach is to use a discretized
version of the \textsc{Kohn}--\textsc{Strang} algorithm \cite{Kohn1986a}
and in Section \ref{sec:Kohn-Strang} we show that this yields a partially
satisfying answer to Question \ref{qu:KMS}. We rely on the convergence
of approximations to the rank-one convex envelope,
which were proved in \cite{Bartels2004, Dolzmann1999, Dolzmann2000b,
  Oberman2017}, see also \cite{Zhang2008} for particular examples.
We remark that the related problem of calculating the rank-one convex hull of
a set still remains poorly understood, see \cite{Angulo2018} and the
references therein. 

To conclude the introduction, let us return to Question
\ref{qu:Morrey} and discuss the numerical results presented in this
note. We randomly sample deformations given by the sum of $N$ plane
waves, for some $N\in \{3,4,5\}$; the cases $N=1,2$ are not
interesting and for $N\geq 6$ we already have that $\# [\D u]\geq 64$,
so the space of functions $g\colon [\D u]\to[-1,1]$ becomes very high-dimensional.
Our approach finds many counterexamples, similar to the ones in \cite{Sverak1992a},
when $m\geq 3$; however, and despite sampling thousands of different
deformations, none where found
when $m=2$. In this last case,
we observe that on average it is easier to check that a given function
does not yield a counterexample to Question \ref{qu:Morrey} as $N$ increases.
We give a basic heuristic explanation of these findings: for
plane wave expansions, the rank-one geometry of the set $[\D u]$ is
drastically different in the cases $m=3$ and $m=2$ and, in the latter,
the geometry becomes much richer as $N$ increases. Our considerations
are inspired by the very interesting results of
\textsc{Sebestyén}--\textsc{Székelyhidi} \cite{Sebestyen2017}, where
the authors tap into this structure
to prove that no counterexamples arise when $m=2$ and $N=3$,
see also \cite{Pedregal1998}.

\begin{para*}{Acknowledgements}
  A.G. would like to thank Vladimír \v Sverák for a very interesting
  conversation about Morrey's problem and, in particular, for the
  suggestion of conducting a numerical experiment roughly along the
  lines of the one described in Section \ref{sec:experiment}. A.G. would also like to
  thank Jan Kristensen for countless hours of discussion about the
  problems addressed here. A.G. was sponsored by
  [EP/L015811/1]. R.TdC. was sponsored by [EP/L016516/1].
\end{para*}

\section{Deciding whether a measure is a laminate}\label{sec:Kohn-Strang}

In this section we discuss Question \ref{qu:KMS}: throughout $\nu$ is
a fixed probability measure with support on a finite set $K\subset
\R^{m\times n}$. In this section we use a discretized version of the \textsc{Kohn}--\textsc{Strang}
algorithm to show the following:

\begin{prop}\label{prop:semidecidable}
Let $\nu$ be a probability measure supported in a finite set of points
in $\R^{m\times n}$. The problem of deciding whether $\nu$ is a
laminate is semidecidable, i.e.\ there is an algorithm which
terminates in finite time with a positive answer if $\nu$ is a laminate.
\end{prop}

To prove this, we will resort to \textsc{Pedregal}'s Theorem \cite{Pedregal1993}: $\nu$ is a laminate
if and only if
\begin{equation}
\label{eq:laminatetest}
f^\tp{rc}(\overline \nu)\leq \langle \nu, f\rangle, 
\hs \overline \nu \equiv \langle \nu, \tp{id}\rangle
\end{equation}
for all continuous $f\colon \R^{m\times n}\to \R$, where $f^\tp{rc}$
denotes the rank-one convex envelope of $f$.

\begin{lema}\label{lema:basiclema}
If $\nu$ is not a laminate there is $g\colon K\to [-1,1]$ such that, for $0< \delta<c(K,n,m)$ small enough, the continuous function $f_\delta\colon \R^{m\times n}\to [-1,2]$, defined by
\begin{equation}
\label{eq:deffdelta}
f_\delta(A) \equiv \begin{cases}
g(A_0)+\frac{2-g(A_0)}{\delta} |A-A_0| & \tp{if } |A-A_0|\leq \delta \tp{ for some } A_0 \in K \\
2 & \tp{otherwise},
\end{cases}\end{equation}
satisfies $f^\tp{rc}_\delta(\overline \nu)>\langle \nu, f^\tp{rc}_\delta \rangle $.
\end{lema}
Although this is not needed, note that $\lim_{\delta\to 0}
f_\delta=f_0\equiv g \mathbbm 1_K + 2\times \mathbbm 1_{\R^{m\times n}\exc K}$ pointwise.

\begin{proof}
Since $\nu$ is not a laminate, there is $\tilde g\colon \R^{m\times
  n}\to \R$ rank-one convex and such that $g(\overline \nu)> \langle
\nu,g\rangle$; by scaling, we can assume that $\tilde
g([-2,2]^{mn})\subseteq [-1,1].$  Since $\tilde g$ is rank-one convex it is locally Lipschitz and, see \cite{Ball2000},
$$\tp{Lip}(\tilde g, [-1,1]^{mn})\leq \min\{m,n\} \tp{\,osc}(\tilde g, [-2,2]^{mn})\leq 2 \min\{m,n\}.$$
Let us take $g= \tilde g|_K$, $\delta\leq \frac 1 2 \min\{|X_1-X_2|: X_1, X_2 \in K\}$ so that $f_\delta$ is well-defined and, in addition, we require that $\delta<(2\min\{m,n\})^{-1}$. Thus, for $A\in B_\delta(A_0)$ and $A_0 \in K$, 
$$\tilde g(A) \leq \tilde g(A_0)+ 2 \min\{m,n\} |A-A_0| \leq \tilde g(A_0)+ \frac 1 \delta |A-A_0| \leq f_\delta(A).$$
This shows that $ f_\delta\geq \tilde g$; since $g$ is rank-one convex, also $f_\delta^\tp{rc}\geq \tilde g$ and so 
$$ f^\tp{rc}_\delta (\overline \nu)\geq \tilde g (\overline \nu) >
\langle \nu, g\rangle= \langle \nu, f_\delta^\tp{rc}\rangle. \qedhere$$
\end{proof}

%Let us suppose that one does find $g\colon K_N \to [-1,1]$ such that $f$, defined as in (\ref{eq:deff}), does not satisfy Jensen's inequality with respect to $\nu$. A priori, this does not imply that $\nu$ is not a laminate, since $f$ is not finite; nonetheless, and due to the specific structure of $\nu$, in most cases we can still show that $\nu$ is not a laminate:

%\begin{prop}
%Assume that (\ref{eq:counterex}) holds for some choice of $g\colon K_N \to \R$ and fix $N \geq nm$. For a.e.\ $(a_i, n_i, c_i)\in \R^{m+n+1}$, $i=1, \dots, N$, the measure $\nu$ is not a laminate.
%\end{prop}
%
%\begin{proof}
%Outside of a null set, since $N\geq mn,$ $\tp{conv}(K_N)$ has non-empty interior.
%
%Take $\delta>0$ small; since $\nu$ is supported in the interior of $B_\delta + \tp{conv}(K_N)$ we see that there is, c.f.\ \cite[Theorem 4.4]{Pedregal1993}.
%\end{proof}

Lemma \ref{lema:basiclema} shows that in order to decide whether $\nu$
is a laminate one has to explore the finite-dimensional space of
functions $g\colon K\to [-1,1]$. In order to compute an approximation of $f^\tp{rc}_\delta$ we use a discrete version of the \textsc{Kohn}--\textsc{Strang} algorithm \cite{Kohn1986a}.

\begin{para}{Algorithm}\label{alg:KS}
We fix $\delta>0$ small enough  so that we do not need to worry about
it; thus we drop the subscript $\delta$. By translation invariance we can assume that $\overline \nu=0$. Then:

\begin{enumerate}[label=\arabic*.]
\item Fix an odd integer $L$, consider the grid
$\mc G_L\equiv \frac{1}{L}\Z^{mn}\cap [-1,1]^{mn}$,
and choose a finite set of directions $\mc D$ consisting of rank-one matrices which are in $\mc G_L$.

\item Set $f^{\tp{rc},0}_{L,\mc D}:=f$ and, for $A\in \mc G_L$,

$$f^{\tp{rc},i+1}_{L,\mc D}(A)=
\min_{\substack{X \in \mc D: A\pm X \in \mc G_L}}\left\{\frac{f^{\tp{rc},i}_{L,\mc D}(A+X)+f^{\tp{rc},i}_{L,\mc D}(A-X)}{2}, f^{\tp{rc},i}_{L,\mc D}(A)\right\}.
$$

We terminate the algorithm if either the maximum difference between iterates stabilizes or $f^{\tp{rc},i}_{L,\mc D}$ satisfies Jensen's inequality with respect to $\nu$.
\end{enumerate}
\end{para}

Let $f^{\tp{rc},i}$ be the $i$-th Kohn--Strang iterate, i.e.\
$f^{\tp{rc},i}$ is defined inductively by $f^{\tp{rc},0}=f$ and
$$f^{\tp{rc},i}(A)=\inf \left\{\lambda f^{\tp{rc},i-1}(X)+ (1-\lambda) f^{\tp{rc},i-1}(Y): \lambda X + (1-\lambda)Y=A, \tp{rank}(X-Y)=1 \right\},$$
where $\lambda$ runs over $(0,1)$. Clearly,
for $A\in \mc G_L$,
$f^{\tp{rc},i}(A)\leq f^{\tp{rc},i}_{L,\mc D}(A)\leq f(A),$
and so if $f^{\tp{rc},i}_{L,\mc D}$ satisfies Jensen's inequality with respect to $\nu$, $f_\delta$ satisfies (\ref{eq:laminatetest}).
Conversely, we have:

\begin{prop}\label{prop:envelopes}
Let $f^{\tp{rc}}_{L,\mc D}\equiv \lim_{i\to \infty} f^{\tp{rc},i}_{L,\mc D}$. Then $f^{\tp{rc}}_{L,\mc D}$ converges uniformly to $f_\delta^\tp{rc}$ as  $L\to \infty$ and as the largest angle between any rank-one matrix and its best approximation in $\mc D$ goes to zero.
\end{prop}

For a proof see \cite{Oberman2017}. Note that we take $0<\delta<c(m,n,K)$ and that $f_\delta$ is continuous, so their results apply.
It is clear that combining Lemma \ref{lema:basiclema} with Proposition
\ref{prop:envelopes},  we deduce Proposition  \ref{prop:semidecidable}.

%It is interesting to compare the situation to what happens when $K$ has no rank-one connections; for a proof see \cite[Theorem 3.3]{Ball1990a}.
%
%\begin{prop}
%Let $K\subset \R^{m\times n}$ be a compact set without rank-one connections and suppose $g\colon \R^{m\times n} \to \R$ is continuous. The function
%$$f(A)=\begin{cases} g(A) & A \in K\\
%+\infty & A \not \in K\end{cases}$$
%is bounded from below and rank-one convex but if there is $A\not \in K$ and $\varphi \in A+W^{1,\infty}_0(\Omega,\R^m)$ such that $\D \varphi \in K$ a.e.\ then $f$ is not quasiconvex at $A$.
%\end{prop}
%
% There are sets $K$ with five points satisfying the assumptions of the proposition \cite{Kirchheim2001b}, but not with four \cite{Chlebik2002}.

\section{Gradient Young measures versus laminates}
\label{sec:deformations}

In this section we adress Question \ref{qu:Morrey}. % Instead of
% considering a general compactly supported gradient Young measure and
% trying to decide whether it is a laminate, we
% focus on a particular class of measures with finite support which nonetheless is
% sufficiently general.
Recall that a function $f\colon \R^{m\times n} \to \R$ is said to be \textit{quasiconvex} if, for all $A\in \R^{m\times n}$, 
\begin{equation}
f(A)\leq \int_{\mb T^n} f(A+\D \varphi(x))\d x \hs \tp{ for all } \varphi
\in C^\infty(\mb T^n, \R^m).\label{eq:qc}
\end{equation}
Equivalently, $f$ is quasiconvex if and only if $f(\overline \nu)
\leq \langle \nu, f\rangle$, where $\nu$ is any compactly supported gradient Young
measure \cite{Kinderlehrer1991, Muller1999a}.
% \subsection{A sufficiently general class of measures}

We want to test the inequality (\ref{eq:qc}) with deformations of the form
\begin{equation}
\varphi (x)=\sum_{i=1}^N a_i s(x \cdot n_i+ c_i),
\label{eq:defphi}
\end{equation}
where $N\in \N$, $a_i \in \R^m, n_i \in \Z^n$ are vectors, $c_i \in \R$ are
phases and $s$ is the 1-periodic sawtooth function, defined for $t\in
[0,1]$ by $s(t)=t1_{[0,1/2]}(t)+(1-t)1_{[1/2,1]}(t)$. The
idea of approximating an arbitrary deformation with a simplified
deformation with the form (\ref{eq:defphi})  is known in the Applied
Harmonic Analysis literature as a ridgelet expansion
\cite{Pinkus2015}. We remark as a somewhat inconvenient fact that
orthonormal ridgelet bases in $L^2$, just like Fourier series, are never unconditional bases in
$L^p$ for $p\neq 2$, although we do not prove this here.

The advantage of an expansion as in (\ref{eq:defphi}) is that, with $h\equiv s'$ being the Haar wavelet,
$$\D \varphi(x)=\sum_{i=1}^N h(x\cdot n_i + c_i)\, a_i \otimes n_i;$$
hence the gradient $\D \varphi$ takes values in a finite set. 
 In our context, considering plane-wave expansions as in
 (\ref{eq:defphi}) is a classical idea, and we are motivated by \textsc{James}' interpretation of \textsc{\v Sverák}'s example \cite{Muller1999a}, see also \cite[\S 31]{Milton2002} and \cite{Pedregal1998, Sebestyen2017}. Moreover, $\varphi$ generates a homogeneous gradient Young measure $\nu$, which takes the form
\begin{equation}
\nu = \sum_{\e \in \{-1,1\}^N} \nu_\e  \delta_{X_\e },\label{eq:defnu}
\end{equation}
where we defined the weights $\nu_\e $ and the matrices $X_\e $ as
\begin{equation}
\nu_\e \equiv |\{x\in \mb T^2:h(x\cdot n_i + c_i)=\e_i, i=1, \dots, N\}|, \hs X_\e \equiv \sum_{i=1}^N \e_i a_i \otimes n_i.
 \label{eq:defX}
\end{equation} 
Note that $\nu_\e $ depends on $n_i$ but not on $a_i$. Furthermore, the measure $\nu$ has barycentre zero. 
% In order to avoid non-degeneracies, we assume that
% \begin{equation}
% \tp{the vectors } n_1, \dots, n_N \tp{ are pairwise linearly independent.}\label{eq:linindep}
% \end{equation}

% The simplest case is when $N=2$ and we have, see  \cite[Lemma 1]{Sebestyen2017}:

% \begin{lema}\label{lema:N=2}
% When $N=2$, $\nu_\e=1/4$ for all $\e\in\{-1,1\}^2$. 
% \end{lema}

% \red{Actually the lemma works as long as $N\leq \max\{m,n\}$. in
%   particular it works for $N=2$.}

% \red{Now in $\R^{m\times 2}$.}

% \begin{lema}
% When $N=3$, order the vectors so that $|n_1|\leq |n_2|\leq |n_3|$ and let $n_3 = k n_1 + l n_2$ for $k,l\in \Q$. Then $\nu_\e=\alpha$ if $\tp{sign}\,\e=1$ and $\nu_\e=\beta$ if $\tp{sign\,}\e=-1$ where, for $c_3\in[0,1/2]$,
% \begin{equation}\label{eq:N=3}
% \a+\beta=\frac 1 4, \hs \a-\beta=\begin{cases}
% \frac{c_3(2c_3-1)}{kl} & k,l \tp{ are odd}\\
% 0 & \tp{otherwise}
% \end{cases}.\end{equation}
% For $c_3\in [1/2,1]$ the same holds by interchanging the roles of $\a$ and $\beta$ in (\ref{eq:N=3}).
% \end{lema}

% \begin{lema}
% When $N=4$, order the vectors so that $|n_1|\leq \dots \leq |n_4|$.
% \end{lema}

% \subsection{N-wave quasiconvexity}

For the sake of conciseness, we introduce the following definition: 
 
\begin{ndef}\label{def:N-wave}
For $N\in \N$, we say that $f\colon \R^{m\times n}\to \R$ is $N$\textit{-wave quasiconvex at zero} if
$$f(0)\leq \sum_{\e \in \{-1,1\}^N} \nu_\e f(X_\e )$$
for all $(a_i, n_i, c_i) \in \R^{m}\times \Z^n\times \R$, where
$\nu_\e $ and $X_\e$ are defined by (\ref{eq:defX}). Moreover, $f$ is \textit{$N$-wave quasiconvex} if, for any $A\in \R^{m\times n}$, the function $f(\cdot - A)$ is $N$-wave quasiconvex at zero.
\end{ndef}

It seems that variants of this notion were studied in
\cite{Kalamajska2003} for $N=3,4$.  By periodicity, in Definition \ref{def:N-wave} we
can assume that $c_1= \dots= c_I=0$ where $I=\min\{n,N\}$. We have:

\begin{prop}\label{prop:anyN}
$f$ is quasiconvex if and only if it is $N$-wave quasiconvexity
for all $N$.
\end{prop}

\begin{proof}
 We prove that if $f$ is $N$-wave quasiconvex at zero it is
 quasiconvex at zero, as the converse is clear.
 We rely on the
following standard fact: for $\varphi\in C^\infty(\mb T^n,\R^m)$ there is a sequence
$\varphi_j$ of the form (\ref{eq:defphi}) which converges to $\varphi$
strongly in $W^{1,\infty}(\mb T^n,\R^m)$.
For a quantitative version of this fact when $m=1$ see e.g.\
\cite{Gordon2002}, although there the authors take several
different functions $s_i$, for $i=1, \dots N$, instead of a fixed
sawtooth function; regardless, any $s_i$ can be approximated by scaled and
translated copies of $s$. The general case $m>1$ follows by
straightforward argument and we omit it.

Let $\nu_j$ be the gradient Young measure generated by the deformation
$\varphi_j$;
by assumption,
$$f(0)=f(\overline \nu_j)\leq \langle \nu_j, f\rangle=\int_{\mb T^n}
f(\D \varphi_j)\d x.$$
Since $\varphi_j\to \varphi$ in $W^{1,\infty}(\mb T^n,\R^m)$, we see
that $f(0)\leq \int_{\mb T^n} f(\D \varphi)\d x.$ Thus $f$ is quasiconvex at zero.
\end{proof}

The following theorem gathers several results from the literature.

\begin{teo}\label{teo:Nwave}
  $N$-wave quasiconvexity has the following properties:
  \begin{enumerate}
   \item \label{it:1-wave} 1-wave quasiconvexity is equivalent to rank-one 
convexity;
\item \label{it:2-wave} 2-wave quasiconvexity is equivalent to rank-one convexity;
\item \label{it:sverak}if $m\geq 3, n\geq 2$ then 3-wave quasiconvexity is different from rank-one convexity and is a nonlocal property;
\item \label{it:sze} if $m=n=2$ then 3-wave quasiconvexity is implied by rank-one convexity.

  \end{enumerate}
\end{teo}

\begin{proof}
\ref{it:1-wave} follows straightforwardly, \ref{it:2-wave} follows by
\cite[Lemma 2.1]{Sebestyen2017},
\ref{it:sverak} follows from the example in \cite{Sverak1992a}
together with an adaptation of the arguments in
\cite{Kristensen1999a} and \ref{it:sze} is the main result of \cite{Sebestyen2017}.
\end{proof}

\section{Counting rank-one connections}\label{sec:counting}

The behaviour of gradients of maps changes dramatically from
the higher-dimensional to the planar case \cite{Faraco2008, Iwaniec2002a,
  Kirchheim2008}.
One of the basic explanations for this difference is that the relative size of the cone
$$\Lambda\equiv \{A\in \R^{m\times 2}: \tp{rank}\,A\leq 1\}$$
is much larger
when $m=2$ than when $m\geq 3$: for instance, it separates the
matrix space into two components in the former case.

The previous insight is also relevant towards the goal of
understanding the behaviour of the particular deformations of Section \ref{sec:deformations}.
In fact, the proof of Theorem \ref{teo:Nwave}\ref{it:sze} in
\cite{Sebestyen2017} also explores the fact that $\Lambda$ is large: using arguments somewhat in the spirit of
\cite{Szekelyhidi2005}, the abundance of rank-one
connections is used to build complicated laminates supported in the 3-cube
$\{X_\e\}_{\e\in \{-1,1\}^3}$. In view of Proposition \ref{prop:anyN}
it is natural to ponder what can be said for a general $N>3$.

In this section our goal
is to roughly quantify the number of rank-one connections
between points in the lamination hull of the $N$-cube. Our observations are merely heuristic, i.e.\ we do not
provide any proofs, and they are the consequence of analysing
thousands of computer-generated random configurations.

% For a particular choice of vectors and phases $(a_i,
% n_i, c_i)_{i=1}^N$ we write $K_N$ for the set of vertices of
% the corresponding
% $N$-cube, i.e.\ $K_N\equiv \{X_\e\}_{\e\in \{-1,1\}^N}$. 

For a given choice of matrices $X_\e$ as in
(\ref{eq:defX}), let us write $$K_N\equiv \{X_\e:\e\in
\{-1,1\}^N\}\subset \R^{m\times 2}, \qquad Q_N\equiv [-1,1]^N\subset \R^N.$$
We can visualise $K_N$ as the
vertices of the $N$-cube $Q_N$ by considering the map $X_\e\mapsto
\e$. Note, however, that for $N>2m$
the map $\e\mapsto X_\e$ cannot be an embedding.

Let us denote by $K_N^{\tp{lc},i}$ the usual $i$-th lamination
convex hull of $K_N$, see 
\cite{Muller1999a} for the definition. Since the
edges of the cube correspond to rank-one segments, it is clear
that, under the above identification, $K_N^{\tp{lc},i}$ contains the
$i$-skeleton of $Q_N$: for instance, $K_N^{\tp{lc},1}$ contains
the edges of the cube, $K_N^{\tp{lc},2}$ contains the faces, and so
on.

We say that $X_\e$ and $X_{\e'}$ are neighbours if $\e$ and $\e'$
are adjacent vertices in $Q_N$. Generically,
each vertex $X_\e$ is rank-one connected only to its $N$
neighbours and thus $K_N^{\tp{lc},1}$ is in fact the 1-skeleton of the
$N$-cube, i.e.\ it consists of the vertices and the edges $E_N\equiv
K_N^{\tp{lc},1}\exc K_N$; note that each edge is an open segment
parallel to a rank-one line.
% =\{E_1,\dots, E_{N2^{N-1}}\}

We now want to compare $K_N^{\tp{lc},2}$ with the 2-skeleton of the
$N$-cube. We call a
rank-one connection  \textit{trivial} if it exists in the 2-skeleton
of the $N$-cube. A vertice is trivially connected to the $N$
edges that have that vertex as one of their endpoints. An edge, which
we write in the form $\{(\e_1, \dots, \e_{i-1}, t, \e_{i+1}, \dots, \e_N): t\in [0,1]\}$, is
trivially connected to the $N-1$ edges that arise by flipping the sign
of one of the $\e_j$, for $j \neq i$.

Let us consider a fixed deformation. We count, for each vertice and
each edge, the
number of non-trivial edges to which it is rank-one connected; thus we
find two vectors of natural numbers, one with length $N$, the other
with $N 2^{N-1}$.
For these vectors we calculate their mean deviation. Finally, sampling
randomly thousands of deformations, we get approximate values for the
average number of connections, see Tables \ref{table:1} and
\ref{table:2}.

\begin{table}[htbp]
    \centering
    \begin{tabular}[h]{|r|c|c|}
      \hline
      & $m=2$ & $m=3$ \\
      \hline
      $N=3$ & 0.95 (0.47) & 0 (0)  \\\hline
      $N=4$ & 4.79 (1.57) & 0 (0)  \\\hline
      $N=5$ & 15.59 (3.65) & 0 (0) \\\hline
      $N=6$ & 41.70 (8.31) & 0 (0)\\\hline
    \end{tabular}
    \vspace{0.3cm}
    \captionsetup{width=.75\linewidth}
    \caption{Average number (and average mean deviation) of the number of
        non-trivial edges to which a vertice is rank-one connected.}
  \label{table:1}
  \end{table}
  
    \begin{table}[htbp]
    \centering
    \begin{tabular}[h]{|r|c|c|}
      \hline
      & $m=2$ & $m=3$ \\
      \hline
      $N=3$ & 2.90 (0.63) & 0 (0)  \\\hline
      $N=4$ & 12.50 (2.39) & 0 (0)  \\\hline
      $N=5$ & 36.78 (7.06) & 0 (0)  \\\hline
      $N=6$ & 92.17 (18.28)  & 0 (0) \\\hline
    \end{tabular}
    \vspace{0.3cm}
    
    \captionsetup{width=.75\linewidth}
    \caption{Average number (and average mean deviation) of the number of
        non-trivial edges to which an edge is rank-one connected.}
     \label{table:2}
   \end{table}

   \begin{remark}\label{rem:tables}
     We would like to make a few points concerning Tables
     \ref{table:1} and \ref{table:2}:
     \begin{enumerate}
     \item\label{it:probabilistic} The values obtained should be
       understood in a probabilistic sense: it is not true that, when
       $m=3$, there are never non-trivial connections. In fact, if we
       randomise vectors $a_i\in (\Z\cap[-L,L])^3, n_i\in
       (\Z\cap[-L,L])^2$ with $L$ a small number, say $L=5$, then we
       find non-trivial rank-one  connections in many of the
       corresponding configurations.
     \item The low average mean deviations in the tables show that the
       connections are not concentrated in a few vertices or edges;
       see also Figure \ref{fig}.
     \item When $m=2$, an increase in $N$ also increases the number of
       connections dramatically.  Thus, although the set $K_N$ becomes
       exponentially more complicated as $N$ increases, the geometry
       of its rank-one lines also becomes much richer.
     \end{enumerate}
   \end{remark}

\begin{remark}
Rank-one lines are very fragile: even if sometimes rank-one connections exist, they are
easily destroyed by small perturbations \cite{Kirchheim2001}. It is therefore more
appropriate to consider the
rank-one convex hull, which is often much larger than the
lamination convex hull, albeit it is also much more difficult to
calculate.
\end{remark}

What we find the most remarkable about Tables \ref{table:1} and \ref{table:2} is not
the fact that there are almost no rank-one connections when
$m=3$ but rather that there are so many connections when
$m=2$. Thus, in low-dimensions, simple lamination seems to be a viable
option to produce very complex gradients.
We believe that Tables
\ref{table:1} and \ref{table:2}
can be taken as partial
evidence towards a positive answer to Question \ref{qu:Morrey} when
$m=n=2$.

\begin{figure}[htbp]
  \centering
  \begin{subfigure}[b]{0.4\textwidth}
    \centering
    \includegraphics[width=0.9\textwidth]{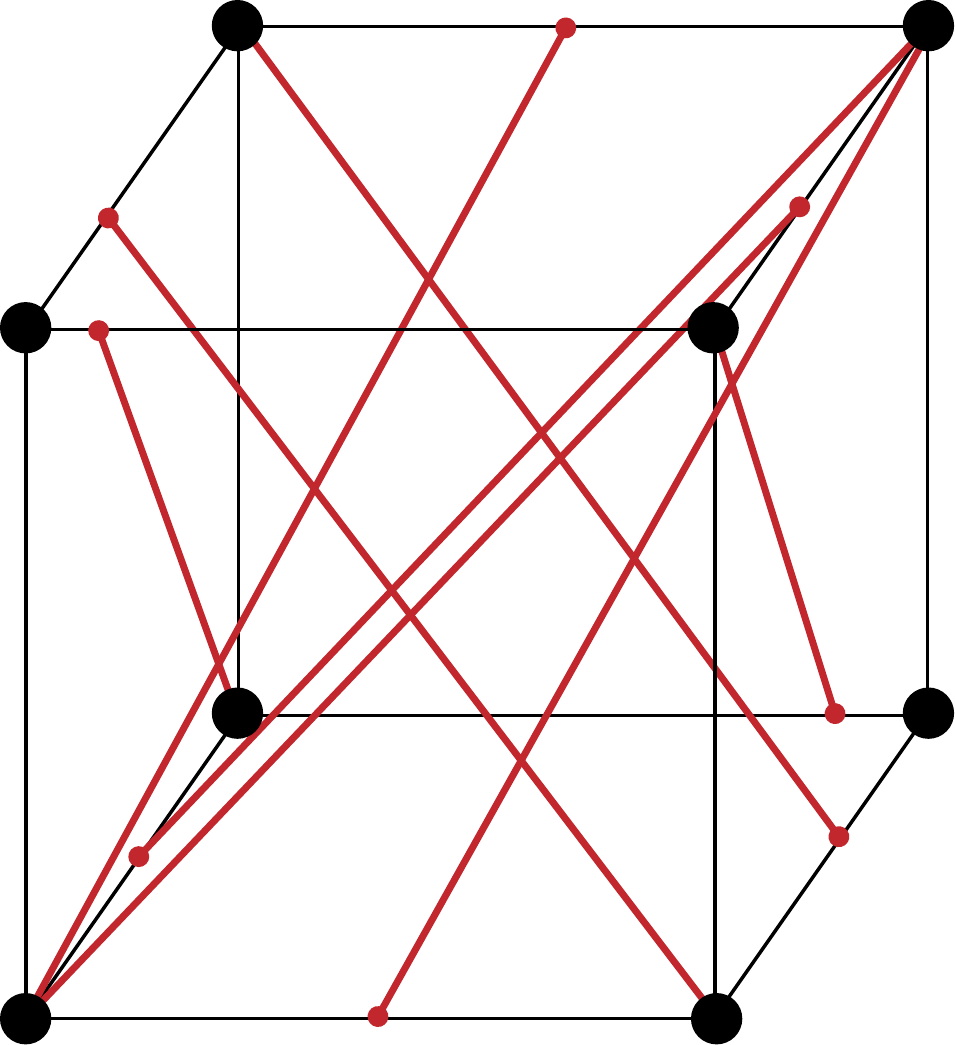}
    \caption{Non-trivial vertice-edge connections.}
  \end{subfigure}%
  ~
  \begin{subfigure}[b]{0.4\textwidth}
    \centering
    \includegraphics[width=0.9\textwidth]{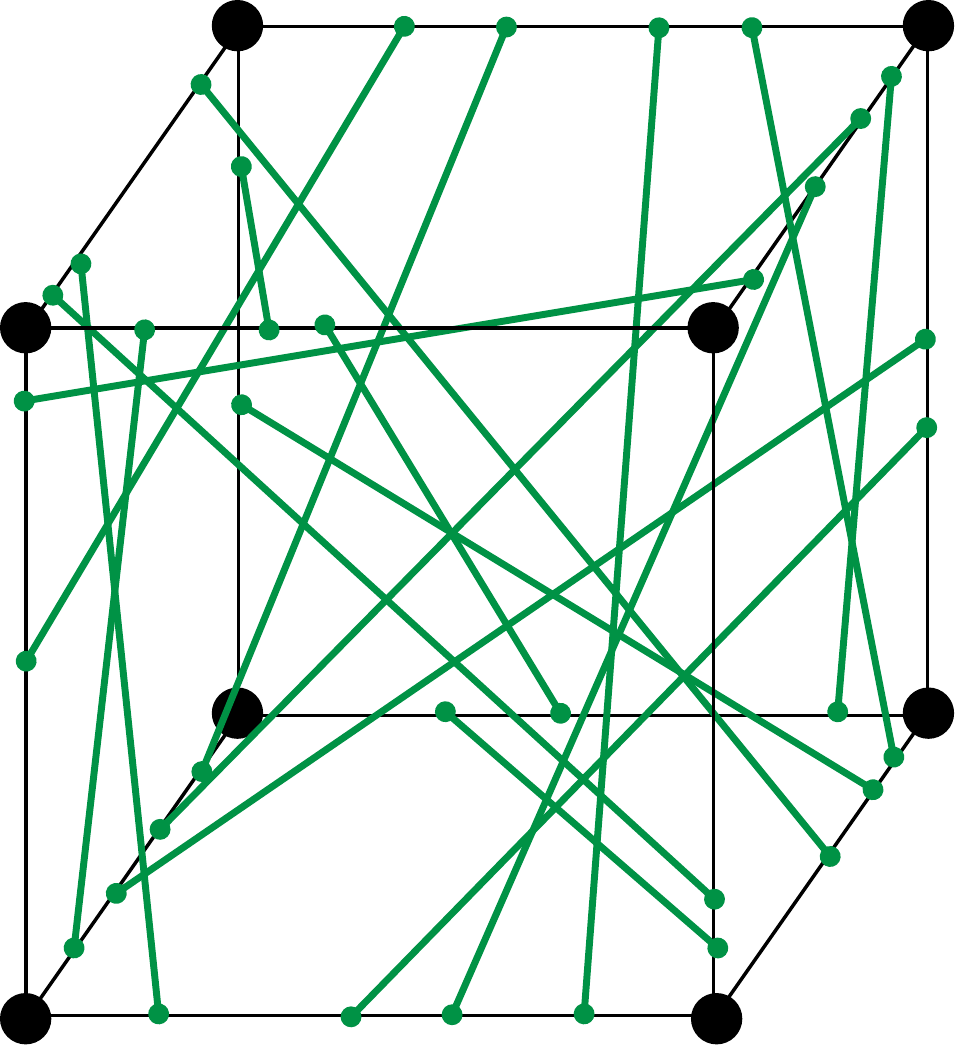}
    \caption{Non-trivial edge-edge connections.}
  \end{subfigure}%
   \captionsetup{width=.8\linewidth}
  \caption{Depiction of a ``typical'' configuration when $N=3$ and
    $m=2$, with an average of 1 vertice-edge connection per vertex and
    3 edge-edge connections per edge. Each line denotes the existence
    of at least one point in the edge which is rank-one
    connected.}
\label{fig}
\end{figure}

% We believe that the abundance of rank-one connections
% when $m=2$ should not be seen as a mere consequence of the relative
% largeness of $\Lambda$, but also of the constraint imposed on the maps
% by the low-dimensionality of the target.

\section{Numerical search for counterexamples to Morrey's problem}
\label{sec:experiment}

In this section we report on numerical experiments which bring
together Sections \ref{sec:Kohn-Strang} and \ref{sec:deformations}. Our goal was
to find numerical evidence towards a resolution of Question
\ref{qu:Morrey}.

We set $n=2$ and run the following algorithm:

\begin{para}{Algorithm} \label{alg:numerical-search}

Fix $L,N\in \N$, with $L$ odd and sufficiently large, and a threshold $\gamma \in[0,1]$. Set  $\mc G_L\equiv \frac{1}{L}\Z^{2m}\cap [-1,1]^{2m}$ and choose a finite set of directions $\mc D$ consisting of rank-one matrices which are in $\mc G_L$. Then:

\begin{enumerate}[label=\arabic*., ref={\arabic*}]
\item \label{it:1n} Randomly generate a set of directions,
  $(n_i)_{i=1}^N$ in $([-L,L]\cap \Z)^2$. Check that $n_i$ and $n_j$ are linearly independent for $i\neq j$; if not, repeat the previous instruction.
\item \label{it:1c} For each $(n_i)_{i=1}^N$, randomly generate a set
  of phases, $(c_i)_{i=1}^N\in\mathbb{R}$. The set $(n_i,
  c_i)_{i=1}^N$ determines the weights, $(\nu_\e)_\e$, at each point
  in the support of the measure, see (\ref{eq:defX}).
\item \label{it:1a} Randomly generate a set of vectors $(a_i)_{i=1}^N$
  in $(\frac{1}{L} \Z \cap [-1,1])^m$. Check that $a_i\neq 0$ for all $i$ and that the matrices $X_\e$ where the measure is supported, defined in (\ref{eq:defX}) in terms of $(a_i, n_i)_{i=1}^N$, are in $\mc G_L$; if not, repeat the previous instruction.
\end{enumerate}

\noindent Repeat Step \ref{it:1n} a number $M_n$ of times; for each of
those, repeat Step \ref{it:1c} $M_c$ times; and for each
$(n_i,c_i)_{i=1}^N$, generate $M_a$ different sets $(a_i)_{i=1}^N$ by
\ref{it:1a}. We thus obtain $M_\nu\equiv M_n\times M_c\times M_a$ sets 
 $(a_i, n_i, c_i)_{i=1}^N\in \Z^m\times \Z^2\times \R$, each of which
 defines a measure $\nu$ supported on $\mc G_L$, see
 (\ref{eq:defnu}). Then, for each such $\nu$, we execute the following:

\begin{enumerate}[label=\arabic*., ref={\arabic*}]
\setcounter{enumi}{3}
\item\label{it:3} Randomly generate vectors in $g\in [-1,1]^{2^N}$ and, for each such vector, define a function $f$ as in (\ref{eq:deffdelta}).

\item \label{it:4} Apply the Kohn--Strang algorithm, as described in
  Algorithm \ref{alg:KS}, to calculate the approximation $f_{L,\mc
    D}^\tp{rc}(0)$ of $f^\tp{rc}(0)$.

  \item Check whether $f_{L,\mc D}^\tp{rc}(0)-\langle \nu,
    g\rangle>\gamma$.
If so, pick another measure of those generated in Steps \ref{it:1n}-\ref{it:1a} and 
go back to Step \ref{it:3}.
    If not, and if this step hasn't been performed more than $M_g$
  times, using the same measure $\nu$, go back to Step \ref{it:3}.
\end{enumerate}

The measure $\nu$ is \textit{suspicious} if it seems that
Jensen's inequality fails, i.e.\ if at least one $g$ generated in Step
\ref{it:3} is such that $f_{L,\mc D}^\tp{rc}(0)-\langle \nu,g\rangle >0$.  Suspicious measures are further examined:
\begin{enumerate}[label=\arabic*., ref={\arabic*}]
\setcounter{enumi}{6}
\item \label{it:5} For each suspicious pair $(\nu,g)$, make the
  changes $(L,\mc D)\mapsto (L',\mc D')$, where $L'=2L-1$ and $\#\mc
  D\leq \# \mc D'$, and rerun Step \ref{it:4}. Repeat the previous instruction as needed.
    
\end{enumerate}
\end{para}

\begin{remark}
Note that the the parameter $\gamma$ ensures that, in Step \ref{it:3},
one keeps looking for $g$'s until one finds a sufficiently suspicious
measure; we typically took $\gamma =0.1$ and we note that in
\textsc{\v Sverak}'s example $f^\tp{rc}(0)-\langle \nu, f\rangle \approx
\frac 1 4$. In fact, suppose that $0<f_{L,\mc D}^\tp{rc}(0)-\langle
\nu,g\rangle\ll 1$; when refining the approximation of $f^\tp{rc}(0)$ as in Step \ref{it:5} it is likely that one finds $f_{L',\mc D'}(0)-\langle \nu,g\rangle <0$ and indeed this has often happened in our calculations. 
\end{remark}

We implemented Steps \ref{it:1n} and \ref{it:1c} of
Algorithm~\ref{alg:numerical-search}, which determine the weights in
the measure \eqref{eq:defnu}, in Mathematica, as it is well suited to
computing $\nu_\e$ as given by \eqref{eq:defX}.  We note that, due to
the complexity of this computation, we were unable to apply our
algorithm to look for counterexamples with $N\geq 6$. Moreover, for
$N=3$, it follows from the work of
\textsc{Sebestyén}--\textsc{Székelyhidi} \cite{Sebestyen2017} that the
admissible weights form a line segment in $\R^{8}$, so it is enough
to look for counterexamples at the endpoints. Only for $N=4,5$ do we,
\textit{a priori}, actually require $N_n\times N_c$ to be large in
order to have a good sampling of the parameter space. 

The bulk of Algorithm~\ref{alg:numerical-search}, i.e.\ Steps
\ref{it:1a}-\ref{it:5}, was implemented in the C programming
language. Our implementation is quite fast for $m=2$:
for instance with $L=25$, $\#\mc{D}=64$ and $M_g=50$, it typically
takes around 3 minutes to perform Steps \ref{it:3}-\ref{it:4}, even
when Step \ref{it:4} is performed the maximum number of times. For
$m>2$, the algorithm has a very large computational cost: for instance,
with $L=19$ and $\#\mc{D}=168$, it typically takes around 13 hours to
perform Steps \ref{it:3}-\ref{it:4} a number $M_g=50$ times. We
remark that in this case the number of points in the grid is approximately $47\times 10^6$.

\subsection{The case $m=2$}

For $m=2$ we considered deformations given by sums of $N$ plane
waves with $N=3,4,5$.

For $N=3$, we verified numerically the analytical result of
\cite{Sebestyen2017}. Using a gridsize of $L=25$, we selected a total
of 210 measures and randomised 50 different functions $g$, which were
rank-one convexified using $\# \mc{D}=64$ rank-one directions, see Table~\ref{tab:numerical-m2}. About 5\% of the pairs $(\nu,g)$ were found to be suspicious, though none above the threshold $\gamma=0.1$. Upon rescaling the grid to $L'=49$ and increasing the set of rank-one directions to $\# \mc{D}'=256$, all but one of these pairs was shown to satisfy Jensen's inequality; the remaining potential counterexample was ruled out by rescaling the grid again to $L'=97$ and increasing $\# \mc{D}'=784$.

It is for $N=4,5$, where Question~\ref{qu:Morrey} is open, that our
results are most interesting. As the structure of the weights in these
cases is unknown, we consider a much larger set of measures,
around 1000, in our numerical tests; we have also increased
the maximum number of functions $g$ to test to $100\times 2^N$, see
Table~\ref{tab:numerical-m2}. We have found that, when compared to a
run for $N=3$ with the same $L$ and $\mc D$, in the case $N=4,5$
there is a drastic
decrease in the percentage of suspicious measures initially flagged by
Algorithm \ref{alg:numerical-search}: when using a gridsize of $L=25$ and $\# \mc{D}=64$
rank-one directions, for example, only 0.06\% of the pairs $(\nu,g)$
are found suspicious when $N=4$ and none are flagged in this way when
$N=5$. From the point of view of our algorithm, Jensen's inequality is
clearly easier to verify as $N$ increases, at least within the range
of $N$ we test, which could be explained by the increase in
size of the 2nd lamination convex hull, c.f.\
Section~\ref{sec:counting}.
% We also note that the quality of the potential counterexamples does not increase when $N=4$:
None of the pairs $(\nu,g)$ flagged as suspicious was found to be a counterexample
after rescaling the grid to $L'=49$ and increasing the set of rank-one
directions to $\# \mc{D}'=256$.
We also tested configurations generated randomly in finer grids, having obtained
identical results to the case $L=25$.

To summarize: after testing thousands of randomly generated measures and hundreds of randomly generated functions, we have not found any counterexamples to Question~\ref{qu:Morrey}.

\begin{table}[htbp]
\centering
\begin{tabular}{|c|c|c|c||c|c|}\hline
$N$ & $M_n$ & $M_c$ & $M_a$ &  $M_\nu$ & $M_g$\\ \hline 
3   & 7     & 1     & 30    &  210     & 50   \\ \hline 
4   & 7     & 7     & 20    &  980     & 160  \\ \hline 
5   & 7     & 7     & 20    &  980     & 320  \\ \hline 
\end{tabular}
\caption{Parameter space sampled in numerical experiments with $m=2$.}
\label{tab:numerical-m2}
\end{table}

\subsection{The case $m=3$}

For $m=3$ and $N=3$, let us conseider directions $(n_1,n_2,n_3)$ which are non-degenerate in the sense that, for some choice of phases, there is $\e\in \{-1,1\}^3$ with $\nu_\e\neq \frac 1 8$. It follows from the example in \cite{Sverak1992a} that, with probability one, any such measure is a counterexample to Question~\ref{qu:Morrey}.
Due to the high computational cost of Algorithm~\ref{alg:numerical-search} for $m=3$, which severely limits our ability to explore the parameter space, we decided to focus only on $N=3$ and attempt to recover these analytic results.

With a grid of size
$L=19$ and $\#\mc D = 168$, over the course of
% \red{13h times $M_\nu$}
two weeks, we tested around 30 measures corresponding to $3$-wave deformations. All but one measure was found to be suspicious and around 90\% of the measures were found to be sufficiently suspicious, in the sense that there was one rank-one convexified function for which Jensen's inequality failed by a margin superior to the threshold of $\gamma =0.1$. We were unable to verify how many of our candidate counterexamples would survive after rescaling the grid (Step~\ref{it:5} of Algorithm~\ref{alg:numerical-search}), as those computations would take around a month per measure. However, these results are in agreement with what is known analytically for $N=3$, further validating our implementation of Algorithm~\ref{alg:numerical-search}.

\subsection{The case $m> 3$}

It is interesting to consider  \textsc{Grabovsky}'s
example \cite{Grabovsky2018} of a rank-one convex, non quasiconvex
function $G\colon\R^{8\times 2}\to \R$. $G$ is quasiconvex at zero,
although not at the point $\tp{Id}_{\mb H^2}\equiv e_1\otimes e_1+e_5\otimes e_2$. However,
the paper \cite{Grabovsky2018} does not give an explicit deformation falsifying
the quasiconvexity inequality (\ref{eq:qc}); the deformation is only
obtained indirectly through the variational principle for the effective
tensor in periodic homogenization.

Due to its
complicated definition it seems difficult to decide analytically
whether $G$ is $N$-wave quasiconvex at $\tp{Id}_{\mb H^2}$ for specific values of
$N>2$. After testing hundreds of different configurations and finding
no counter-example to Jensen's inequality we are led
to suppose that $G$ is $N$-wave
quasiconvex for $N\leq 5$. We note, however, that there
is a curious similarity between the plane-wave expansions of
Section \ref{sec:deformations} and \cite[equation (2.18)]{Grabovsky2018}.

{\footnotesize
\bibliographystyle{acm}

\bibliography{/Users/antonialopes/Dropbox/Oxford/Bibtex/library.bib}
%\bibliography{/scratch/bibtex/library.bib}
}

\end{document}